\documentclass[12pt,a4paper]{amsart}
\usepackage[foot]{amsaddr}
\usepackage{amsmath}
\usepackage{amsthm}
\usepackage{amssymb}
\usepackage{mathrsfs}
\usepackage{verbatim}
\usepackage{xcolor}
\usepackage[colorlinks=true,linkcolor=blue,citecolor=red]{hyperref}
\usepackage{enumitem}
\usepackage{multicol}
\usepackage{graphicx}
\usepackage{float}
\usepackage{txfonts}
\usepackage[nameinlink]{cleveref}

\theoremstyle{plain}
\newtheorem{theorem}{Theorem}
\newtheorem*{theorem*}{Theorem}
\newtheorem{lemma}[theorem]{Lemma}

\theoremstyle{definition}

\theoremstyle{remark}

\numberwithin{equation}{section}
\numberwithin{theorem}{section}
\allowdisplaybreaks


\begin{document}

\title[ Hardy Inequality ]{Boundary fractional Hardy's inequality in dimension one: The critical case}

\author[Adimurthi]{Adimurthi}
\address{\hskip-\parindent
Adimurthi,  IIT Kanpur and TIFR-CAM, Bangalore}
\email{adimurthi@iitk.ac.in}
\author[Purbita]{Purbita Jana}
\address{\hskip-\parindent
Purbita Jana,  Madras School  of Economics, Chennai, India}
\email{purbita@mse.ac.in}
\author[Prosenjit]{Prosenjit Roy}
\address{\hskip-\parindent
Prosenjit Roy, Department of Mathematics, IIT kanpur,
Kanpur -208016, India.}
\email{prosenjit@iitk.ac.in}

\subjclass[2020]{}

\begin{abstract}
We prove  fractional boundary Hardy's inequality in  dimension one for the critical case $sp =1$.  Optimality of the inequality is obtained for any $p$. The extra logarithmic correction term appears in usual fashion.  We also provide a concrete (workable) example of a sequence of smooth functions  that converges to constant function in $W^{s,p}((0,1))$ for $sp=1$ and $p=2$.
\end{abstract}

\maketitle
\section{Introduction}
Hardy's inequality for both local case and nonlocal case has attracted lot of research over the last few decades.  We refer to the  following articles on the study of local Hardy type inequalities: \cite{adinir,BREZIS1,BREZIS2,cazacu} and references there in for more work on the subject.  Regarding the work on fractional boundary (singularity of the integrand on the left hand side of the inequality is on the boundary)
 Hardy inequality, we refer to the work of Dyda in \cite{DYDA2004},  where the following fractional boundary Hardy's inequality is established under other appropriate assumptions on the domain $\Omega, s , p$:
   \begin{align}\label{fbh}
            \left\| \frac{ u}{\delta_x^s}\right\|_{L^p(\Omega)}\leq C[u]_{W^{s,p}(\Omega)}, \ \ \forall\,u\in C_c^\infty(\Omega), \ \Omega \subset \mathbb{R}^d,
        \end{align}
 where $\delta_x$ denoted the distance of the point $x$ from boundary of $\Omega$. $[u]_{W^{s,p}(\Omega)}$ (or simply, $[u]_{s,p, \Omega}$) denotes the semi norm in usual $W^{s,p}(\Omega)$ space. More precisely, Dyda proved the following theorem:
 \begin{theorem}\emph{[Dyda, 2004]}
     \label{dyaa2004} Let $s \in (0,1), p\geq 1$, then \ref{fbh} holds true in one of the following cases,
     \begin{enumerate}
         \item $\Omega$ is a bounded Lipschitz domain and $sp >1$.
         \item  $\Omega$ is the complement of a Lipschitz domain, $sp \neq 1$ and $sp\neq d$.
         \item $\Omega$ is the set above the graph of a Lipschitz domain from $\mathbb{R}^{d-1}\rightarrow \mathbb{R}$ and $sp \neq 1$.
         \item  $\Omega$ is the complement of a point and $sp \neq d$.
     \end{enumerate}
 \end{theorem}
 The main aim of this article is to   extend point one of  Theorem \ref{dyaa2004} to the critical case $sp=1$ in dimension $1$.  In fact our main result (Theorem \ref{mainresult}) is more general, as the left hand side of our inequality holds true for more general parameter. By critical cases in the above theorem we mean 
 \begin{enumerate}
     \item $sp = 1$ in (1),
     \item  $sp =1$ or $sp=d $ in (2),
     \item  $sp =1$ in (3),
     \item  $sp =d$ in (4).
 \end{enumerate}
 To the best of our knowledge there are only  three results that deals with the critical case for the fourth case of Theorem \ref{dyaa2004}. They are due to Squassina-Nguyen in \cite{Nguyen2018}, \cite{Nguyenmarco2019} and by Triebel-Edmund in \cite{tribel} for the fourth case of Theorem \ref{dyaa2004}.  Again, to the best of our knowledge no other work deals with the critical case for the first three cases in Theorm \ref{dyaa2004}. In Theorem 3.1 of \cite{Nguyen2018} authors proved the following result:
\begin{theorem}\emph{[Nguyen and Squassina, 2018]}\label{Nguyen2018}
    Let $d\geq 1$, $p>1,\,s\in (0,1) ,\tau \geq p, sp=d$.
Then the following hold for some  constant $C>0$:   if $1/\tau+\gamma/d =0,$ then 
        \begin{align*}
            \left\| \frac{|x|^\gamma u}{\log(4R/|x|)}\right\|_{L^\tau(\mathbb{R}^d)}\leq C[u]_{W^{s,p}(\mathbb{R}^d)}, \ \ \forall\,u\in C^1_c(\mathbb{R}^d), \ \textrm{Support} (u) \subset B(0,R).
        \end{align*}
  \end{theorem}
In fact  Squassina-Nguyen's result [in \cite{Nguyen2018}] is much more general than stated above. They have established the full range
Caffarelli-Kohn-Nirenburg [see, \cite{ckn1,ckn2}] inequality for fractional Sobolev spaces. But their result  is a point type (singularity at origin only) fractional Hardy inequality.  Now we introduce our main result. We prove the following fractional boundary Hardy inequality in dimension one:
\begin{theorem}\label{mainresult}Let
$p>1,\,s\in (0,1) ,\tau>0, sp=1$, $\gamma\in  [-s,0)$ .
Then the following inequality holds for some  constant $C>0$:   if $1 + \tau\gamma=0,$ then 
        \begin{align*}
            \left\| \frac{|\delta_x|^\gamma (u- (u)_{(0,1)})}{\log(2/\delta_x)}\right\|_{L^\tau((0,1))}\leq C[u]_{W^{s,p}((0,1))}, \ \ \forall\,u\in W^{s,p}((0,1)),
        \end{align*}
        where $\delta_x = dist(x, (0,1)^c)$ and $(u)_{(0,1)} = \int_0^1u(x) dx.$ When $\tau = p$ the inequality above is sharp, in the sense that one cannot further multiply  the left hand side of the above inequality by a function that tends to infinity as $x \rightarrow 0+$ or $1-$.
  \end{theorem}
    We point out that   ``in dimension one" Theorem \ref{mainresult} is very different from  Theorem \ref{Nguyen2018}, as one can see the appearance of $u - (u)_{(0,1)}$ on the left side of the inequality in our main result which was absent in Theorem \ref{Nguyen2018}. This is because, for the case $sp=1$, one has $W_0^{s,p}((0,1))=W^{s,p}((0,1))$ [see, Theorem 6.76, \cite{leoni}] and hence constant functions are there in $W^{s,p}((0,1))$. Secondly, the semi norm in Theorem \ref{Nguyen2018} is on full Euclidean space, but in our case it is in $(0,1)$. Results similar ( for $\tau = p$ only) to Theorem \ref{Nguyen2018} for the case $sp=d$ was obtained in \cite{tribel} by using interpolation technique.  We could not adopt \cite{tribel} to get Theorem \ref{mainresult}. The method of our proof strictly follows the  line as it is in the proof of Theorem \ref{Nguyen2018} with necessary modifications.   The main reason that our result is restricted to dimension one is that  the equation in Lemma \ref{lemmasob} is scale invariant  of $\lambda$ only in dimension one ($C$ will depend on $\lambda$ for higher dimensions).
    
    To the best of our knowledge this is the first result concerning the fractional boundary Hardy inequality in the critical case. Analogue of Theorem \ref{mainresult} for $sp<1$ is done in \cite{DYDA2004}.  The best $C$ in Theorem \ref{mainresult} is also unknown.  
    
    Available literature on Hardy type inequality is huge and it is very difficult to mention all of them. For generalisation of Hardy type inequalities to Orlicz spaces we refer to the work of \cite{cianchi2021, roy2022,Salort2023, salort2022}. For other related work on the subject we refer to  \cite{adi3,arka, sandeep,brezis4,cazacu2} and the references there in. We refer to \cite{hitch,leoni} for a good reading on fractional Sobolev space.  When $\Omega = \mathbb{R}^d\setminus\{ 0\}$ in Theorem \ref{dyaa2004}, the best constant $C$ in \eqref{fbh}  is studied in \cite{rupert} by finding an appropriate ground state solution. For other works on  fractional boundary Hardy inequality we refer to \cite{abdell2017,dyda1, dyda3,Dyda2022,dyda2, ritva, kufner}. 

  As pointed earlier it is well known that for the case $sp=1$, one has $W_0^{s,p}((0,1))=W^{s,p}((0,1))$.  This implies that $C_c^\infty((0,1))$ (or simply Lipschitz continuous functions with $0$ boundary values) are dense in $W^{s,p}((0,1)).$ One can see that usual cut-off function 
  \begin{equation*}
     \rho_{\epsilon}=  \begin{cases}
     x/\epsilon,  \  \hspace{3mm} x \in (0,\epsilon) \\
     1, \ \hspace{3mm} x \in (\epsilon, 1-\epsilon)\\
     (1-x)/\epsilon, \hspace{3mm} x \in (1-\epsilon,1)
      \end{cases}
  \end{equation*}
does not  converge to the constant function $1$ in $W^{s,p}$ norm. In Theorem \ref{dense} we give an explicit $u_\epsilon$ which works. To the best of our knowledge, explicit workable expression for such $u_\epsilon$ is not written down in literature explicitly. Appearance of $\log$ type of cut off is present in our $u_\epsilon$ as well, which is somewhat related to the appearance of $\log$ term in Theorem \ref{mainresult}.\smallskip 

The paper is organised as follows: In the next section we provide some necessary preliminaries in form of lemma and notations to be used in this article. Third section contains the proof of the density theorem [Theorem \ref{dense}] as some part of its calcualtion will be used in the proof of Theorem \ref{mainresult}.  The last section contains the proof of the main theorem,  Theorem \ref{mainresult}.

  \section{Preliminaries}
Throughout this article for a  measurable set $\omega \subset \mathbb{R}$, $(u)_\omega$ will denote the average of the function $u$ over $\omega$, that is 
$$(u)_\omega := \frac{1}{|\omega| }\int_\omega u(x)dx=  \fint_\omega u(x)dx.$$
Whereas for $s \in (0,1), p >1$, $$|u|_{p,\omega}^p := \int_{\omega} |u|^p, \ \textrm{and}  \ |u|_{s,p,\omega}^p:= \int_{\omega} |u|^p + \int_{\omega}\int_{\omega} \frac{|u(x)-u(y)|^p}{|x-y|^{1+sp}}dxdy$$ will denote the usual $L^p(\omega)$ and $W^{s,p}(\omega)$ norm respectively. $|E|$ will simply denote the Lebesgue measure of a set in $\mathbb{R}^d.$ Throughout this article $C >0$ will denote a generic constant that may change from line to line. \smallskip

\noindent\textbf{Poincar\'e Inequality}  \ Let $a>0$, $p \geq 1$, $s\in (0,1)$ then for each $u \in C^1{([0,a])}$, one has 
$$ \int_{0}^a |u(s)-(u)_{(0,a)}|^p ds \leq C(s,p,a)[u]_{s,p,(0,a)}^p.$$
\smallskip 
The following lemma is Lemma 2.2 in \cite{Nguyen2018} with choice of $a=1 =d$ and $sp=1$.
\smallskip
\begin{lemma}\emph[Sobolev Inequality for the critical case]\label{lemmasob}  \ Let $sp=1$ and $D= (\lambda r,\lambda R)$. For any $\tau >0$ and  $\lambda, r, R >0$,  there exist a constant $C=C(r, R,s, p) >0 $ such that 
   $$\left( \fint_{r\lambda}^{R\lambda}|u(x)-(u)_D|^\tau dx \right)^{\frac{1}{\tau}} \leq C [u]_{s,p, D}.$$
 \end{lemma}  
\begin{lemma}\label{lem1}
Let  $ E = (a,b), F=(b,c) \subset \mathbb{R}$ be disjoint. Then for any $\tau \geq 1 $  one has for some constant $C >0$, $\forall u \in C^1([a,c]),$
$$|(u)_F -  (u)_E| \leq \frac{C (c-a)}{\min\{(b-a),(c-b)\}} [u]_{s,p, E\cup F}.$$
\end{lemma}
\begin{proof} 
Let us start with the expression $|(u)_F -(u)_E|^\tau$.
\begin{multline*}
    \label{4}
    |(u)_F -(u)_E|^\tau = \left|\fint_F \{ u(x) - (u)_{E\cup F} \}dx - \fint_E\{u(x) - (u)_{E\cup F} \} dx\right|^\tau  \\ \leq C\left|\fint_F \{ u(x) - (u)_{E\cup F} \}dx\right|^\tau  + C \left|\fint_E\{u(x) - (u)_{E\cup F} dx \} \right|^\tau .
\end{multline*}
Now applying Jensen's inequality (this is the step where the assumption $\tau \geq 1$ is required) we have 
\begin{equation*}
    \label{5}
  |(u)_F -(u)_E|^\tau  \leq C\fint_F \left|u(x) - (u)_{E\cup F} \right|^\tau dx + C \fint_E\left|u(x) - (u)_{E\cup F} \right|^\tau dx.
\end{equation*}
This implies 
\begin{multline*}
     |(u)_F -(u)_E|^\tau \leq \frac{C}{\min\{|E|,|F|\}}\int_{E\cup F}\left|u(x) - (u)_{E\cup F}   \right|^\tau dx  \\ \leq \frac{C|E \cup F|}{\min\{|E|,|F|\}}\fint_{E\cup F}\left|u(x)- (u)_{E\cup F} \right|^\tau dx.
\end{multline*}
Applying lemma \ref{lemmasob}, with $\lambda = 1, r= a$ and $R=c,$ we have
 \begin{equation*} \label{6} 
    \left( \fint_{E\cup F} |u(x) - (u)_{E\cup F}|^\tau dx \right)^{\frac{1}{\tau}} \leq C [u]_{s,p, E\cup F}, \hspace{3mm} \forall u \in C^1([a,c]). 
\end{equation*}
Finally combining \eqref{5} and \eqref{6} the result follows.
\end{proof}

  \begin{lemma}\emph{[An Inequality]}\label{lem2} \ Let $\Lambda, \tau >1$. There exist $C(\Lambda,  \tau) > 0, $ such that for all $1 < D_\tau < \Lambda$, one has 
  $$\left(|a| +|b|\right)^\tau \leq D_\tau |a|^\tau + \frac{C}{(D_\tau -1)^{\tau -1}}|b|^\tau.$$
  \end{lemma}
\noindent The proof of the above inequality is done in  lemma 2 of \cite{Nguyenmarco2019},  but we give a short here for the completeness of this article.
\begin{proof}
Putting $x= |a|/|b|$, it is sufficient to prove that $$(1+x)^\tau\leq 
 D_\tau x^\tau + \frac{C}{(D_\tau -1)^{\tau -1}}, \hspace{3mm} x >0.$$
  Consider the function $f(x)= (1+x)^\tau - 
 D_\tau x^\tau - \frac{C}{(D_\tau -1)^{\tau -1}}$. The above inequality is trivially true for $x \in (0,1]$  for choice of $C$ large enough. Notice that $f(1) <0$ and $\lim_{x \rightarrow \infty}f(t) = -\infty$. Also there exist unique $x_0 = (D_\tau^{\frac{1}{\tau -1}}-1)^{-1}\in (1,\infty)$, such that $f'(x_0)= 0$ and $f(x_0) <0$. All this together proves the inequality.
\end{proof}

Now we state and give a sketch of the proof of  another elementary lemma that will be used several times later on.

\begin{lemma}
    \label{lem3}
    Let $m_i \geq 0$ and $\lambda \in (0,1]$, then 
    $$\left(\sum_{i \in \mathbb{Z}} m_i\right)^\lambda \leq \sum_{i \in \mathbb{Z}} m_i^\lambda.$$
\end{lemma}
\begin{proof}
The proof is by induction argument. For $i=2$, due to homogeneity, it is sufficient to prove that the function $f(x) = (x+1)^\lambda -x^\lambda -1$ is non-positive on the set $x \geq 0$. But this follows easily after observing $f(0)=0$ and $f'(x) <0 $ for $x >0$. Then assuming the inequality to be true for $i=n$, it is easy to prove that it is true for $i=n+1$ by writing $$\sum_{i=1}^{n+1}m_i = \left( \sum_{i=1}^{n}m_i \right) + m_{n+1}$$ and using the first case. Finally the proof follows by taking limit $n$ to infinity.
\end{proof}

  \section{Proof of Theorem \ref{mainresult}}
The following theorem will be used to proof our main theorem \ref{mainresult}.
  \begin{theorem}
      \label{intermediate} Under the assumptions of Theorem \ref{mainresult}, we have 
      \begin{align*}
            \left\| \frac{|\delta_x|^\gamma u}{\log(2/\delta_x)}\right\|_{L^\tau((0,1/2))}\leq C\left( [u]_{W^{s,p}((0,1))}^p + |u|_{p,(0,1)}^p\right)^{\frac{1}{p}}, \ \ \forall\,u\in C^1([0,1]),
        \end{align*}
        where $\delta_x = dist(x, (0,1)^c).$
     \end{theorem}
     Let us finish the proof of our main theorem using the Theorem \ref{intermediate}, then we will provide the proof of it.\smallskip

    \noindent \textbf{\textit{Proof of Theorem \ref{mainresult}}} \ First observe that it is sufficient to proof the required inequality for functions in $C^1([0,1])$ as it is a dense subset of $W^{s,p}((0,1)).$
 Applying Theorem \ref{intermediate} on $u - (u)_{(0,1)} \in C^1{([0,1])}$,  we obtain 
 \begin{align*}
            \left\| \frac{|\delta_x|^\gamma (u-(u)_{(0,1)}}{\log(2/\delta_x)}\right\|_{L^\tau((0,1/2))}\leq C\left( [u-(u)_{(0,1)}]_{W^{s,p}((0,1))}^p + |u-(u)_{(0,1)}|_{p,(0,1)}^p\right)^{\frac{1}{p}}.
        \end{align*}
        Applying  Poincar\'e inequality and using $$[u-(u)_{(0,1)}]_{W^{s,p}((0,1))}= [u]_{W^{s,p}((0,1))},$$ we get 
    \begin{align*}\label{1k}
            \left\| \frac{|\delta_x|^\gamma (u-(u)_{(0,1)})}{\log(2/\delta_x)}\right\|_{L^\tau((0,1/2))}\leq C [u-(u)_{(0,1)}]_{W^{s,p}((0,1))}, \ \ \forall u\in C^1([0,1]).
        \end{align*}    
        Now independently, notice that using the change of variable $1-x=t$, one has 
        \begin{equation*}\label{k3}
\left\| \frac{|\delta_x|^\gamma (u-(u)_{(0,1)})}{\log(2/\delta_x)}\right\|_{L^\tau((1/2,1))} = \left\| \frac{|\delta_t|^\gamma (u-(u)_{(0,1)})}{\log(2/\delta_t)}\right\|_{L^\tau((0,1/2))}.\end{equation*}
We write
$$\left\| \frac{|\delta_x|^\gamma (u-(u)_{(0,1)})}{\log(2/\delta_x)}\right\|_{L^\tau((0,1))}^\tau = 
\left\| \frac{|\delta_x|^\gamma (u-(u)_{(0,1)})}{\log(2/\delta_x)}\right\|_{L^\tau((0,1/2))}^\tau + \left\| \frac{|\delta_x|^\gamma (u-(u)_{(0,1)})}{\log(2/\delta_x)}\right\|_{L^\tau((1/2,1))}^\tau := A +B.
$$
We can apply directly the previous theorem (Theorem \ref{intermediate})  on the term $A$ to get the required right hand, while an analogous version of Theorem \ref{intermediate} can also be proved for interval $(1/2,1)$ to estimate the term $ B$.. Combining them the proof is completed.
\bigskip 

\noindent \textbf{\textit{Proof of Theorem \ref{intermediate}}} \  Define the set $A_k := \left\{ x \ \big| \ 2^k  \leq x <2^{k+1}\right\}$ for $k \in \{  -1, -2, \cdots \} := \mathbb{Z}^-$. Notice that 
\begin{equation*}
(0,1) = \underset{k=-\infty}{\overset{-1}{\cup}} A_k.
\end{equation*}
Applying Sobolev inequality  for the critical case with 
$r=1, R=2, \lambda = 2^k$ for $k \in \mathbb{Z}^-$ and $D:= A_k$, we have  for some constant $C>0$ (independent of $k$),
\begin{equation}
    \label{sobo}\left( \frac{1}{2^k} \int_{A_k}|u- (u)_{A_k}|^\tau dx\right)^{\frac{1}{\tau}}\leq C [u]_{s,p, A_k}.
\end{equation}
Independently,  using convexity we have for some constant $C>0$,
\begin{equation}
    \label{convex}
    |u|^\tau \leq C \left\{|u-(u)_{A_k}|^\tau +|(u)_{A_k}|^\tau \right\}.
\end{equation}
Using the condition $\tau\gamma = -1$, we obtain 
\begin{equation*}
    \label{1}
\int_{A_k}|x|^{\tau \gamma}|u|^\tau dx \leq  2^{\tau \gamma(k+1)} |A_K|\fint_{A_k}|u|^\tau dx  =  C \fint_{A_k}|u|^\tau dx.
\end{equation*}
Now applying \eqref{sobo} and \eqref{convex},
\begin{equation*}
\int_{A_k}|x|^{\tau \gamma}|u|^\tau dx  \leq  \fint_{A_k}|(u)_{A_k}|^\tau dx + C[u]_{s,p,A_k}^\tau, \hspace{3mm} \forall k \in \mathbb{Z}^-. 
\end{equation*}
Let $k \leq -2$. For each $x \in A_k$, we have $2^k <  (\delta_x =) x <2^{k+1}.$ This implies that $2/x > 2^{-k}$ and hence $\log(2/x) >(-k)\log(2)$. Therefore, we have 
\begin{equation*}
    \label{2}
       \int_{A_k}\frac{|\delta_x|^{\tau \gamma}|u|^\tau}{\log^\tau(2/\delta_x)} dx = \int_{A_k}\frac{|x|^{\tau \gamma}|u|^\tau}{\log^\tau(2/x)} dx \leq C \left\{\frac{|(u)_{A_k}|^\tau}{(-k)^\tau} + \frac{[u]_{s,p,A_k}^\tau}{(-k)^\tau} \right\}.
\end{equation*}
Since $k \in \mathbb{Z}^-$, we have trivially
\begin{equation*}
    \label{2p}
    \int_{A_k}\frac{|\delta_x|^{\tau \gamma}|u|^\tau}{\log^\tau(2/\delta_x)} dx \leq C \left\{\frac{|(u)_{A_k}|^\tau}{(-k)^\tau} + [u]_{s,p,A_k}^\tau \right\}.
\end{equation*}
Summing the above inequality from $m \in \mathbb{Z}^-$ to $-2$, we obtain after using Lemma \ref{lem3} with $\lambda = p/\tau \leq 1$,
\begin{multline}
    \label{3}
    \int_{2^{m}}^{1/2} \frac{|\delta_x|^{\tau \gamma}|u|^\tau}{\log^\tau(2/\delta_x)}dx = \sum_{k=m}^{-2}\int_{A_k}\frac{|x|^{\tau \gamma}|u|^\tau}{\log^\tau(2/x)}dx \leq C \left\{\sum_{k=m}^{-2}\frac{|(u)_{A_k}|^\tau}{(-k)^\tau} +\sum_{k=m}^{-2} [u]_{s,p,A_k}^\tau \right\} \\
    \leq  C\sum_{k=m}^{-2}\frac{|(u)_{A_k}|^\tau}{(-k)^\tau}  + C[u]_{s,p,(0,1)}^\tau.
\end{multline}
Now independently, applying lemma \ref{lem1} with $E = A_k$ and $F= A_{k+1},$ we get for some constant $C>0$ (independent of $k$) 
\begin{equation*}
    \label{7}
    |(u)_{A_{k}} |^\tau \leq  \left|(u)_{A_{k+1}} +C [u]_{s,p, A_k\cup A_{k+1}}\right|^\tau.
\end{equation*}
For $k \in \mathbb{Z}^{-}\setminus \{-1\}$, consider the numbers $D_\tau^k := \left( \frac{-k}{-k - (1/2)}\right)^{\tau -1}  > 1.$ Observe that $D_\tau^k \leq 2.$ Now applying lemma \ref{lem2} with $D_\tau= D_\tau^k, \Lambda = 2^{\tau -1}$ we obtain 
\begin{equation}
    \label{8}
    |(u)_{A_{k}} |^\tau \leq \left( \frac{-k}{-k - (1/2)}\right)^{\tau -1}  \left|(u)_{A_{k+1}} \right|^\tau + C (-k)^{\tau -1}  [u]_{s,p, A_k\cup A_{k+1}}^\tau.
\end{equation}
In the above estimate we have used the fact that 
$$D_\tau^k - 1 \sim \frac{1}{-k}.$$
Dividing   both sides of \eqref{8} by $(-k)^{\tau -1}$, we have 
\begin{equation*}
    \label{9}
    \frac{|(u)_{A_{k}} |^\tau}{(-k)^{\tau -1}} \leq  \frac{(u)_{A_{k+1}}^\tau}{(-k - (1/2))^{\tau -1}}   +  C [u]_{s,p, A_k\cup A_{k+1}}^\tau.
\end{equation*}
Summing the above inequality from 
$k=m$ to $k=-2$, we have 
\begin{equation*}
    \label{10}
  \sum_{k=m}^{-2}\frac{|(u)_{A_{k}} |^\tau}{(-k)^{\tau -1}} \leq \sum_{k=m}^{-2} \frac{(u)_{A_{k+1}}^\tau}{(-k - (1/2))^{\tau -1}}   +  C\sum_{k=m}^{-2}   [u]_{s,p, A_k\cup A_{k+1}}^\tau.   
\end{equation*}
Changing sides, rearranging and re-indexing,   we get 
\begin{equation}
    \label{12}
    \frac{|(u)_{A_{m}} |^\tau}{(-m)^{\tau -1}} + \sum_{k=m+1}^{-2} \left\{ \frac{1}{(-k)^{\tau -1}} - \frac{1}{(-k+(1/2))^{\tau -1}}\right\}|(u)_{A_{k}} |^\tau \leq C(\tau)|(u)_{A_{-1}} |^\tau  + C\sum_{k=m}^{-2}  [u]_{s,p, A_k\cup A_{k+1}}^\tau.  
\end{equation}
Now using the asymptotics, 
$$\frac{1}{(-k)^{\tau -1}} - \frac{1}{(-k+(1/2))^{\tau -1}} \sim \frac{1}{(-k)^{\tau -1}}.$$
Putting this in \eqref{12},
\begin{equation}\label{ss}
\sum_{k=m}^{-2}  \frac{|(u)_{A_{k}} |^\tau}{(-k)^{\tau -1}} \leq C(\tau)|(u)_{A_{-1}}|^\tau  + C\sum_{k=m}^{-2}   [u]_{s,p, A_k\cup A_{k+1}}^\tau \leq C(\tau)|(u)_{A_{-1}}|^\tau  + C   [u]_{s,p, (0,1)}^\tau .
\end{equation}
In the last inequality we have used Lemma \ref{lem3} with $\lambda = p/\tau \leq 1, m_i = [u]_{s,p,A_i\cup A_{i+1}}^\tau$ for $i= m, m+1, \cdots, -2$ and remaining $m_i =0$.
Independently we have after using H\"older's inequality,
\begin{equation}\label{ee}
(u)_{A_{-1}} = 2 \int_{\frac{1}{2}}^1 u(s)ds \leq 2 \int_0^1|u(s)|ds \leq 2 |u|_{p,(0,1)}.
\end{equation}
Combining \eqref{ss}, \eqref{3} together with \eqref{ee},  we get
\begin{equation}\label{m}
   \left( \int_{2^{m}}^{1/2} \frac{|\delta_x|^{\tau \gamma}|u|^\tau}{\log^\tau(2/\delta_x)}dx \right)^\frac{1}{\tau} \leq   \left( C'(\tau) |u|_{p,(0,1)}^\tau + C [u]_{s,p,(0,1)}^\tau \right)^{\frac{1}{\tau}} \leq \left( C''(\tau) |u|_{p,(0,1)}^p + C' [u]_{s,p,(0,1)}^p\right)^{\frac{1}{p}}.
\end{equation}
In the last inequality we have again used Lemma \ref{lem3} (with only two non zero values of $m_i$) with $\lambda = p/\tau \leq 1.$  Finally taking limit $m \rightarrow -\infty$ in \eqref{m} and using monotone convergence theorem we get the desired proof. \bigskip

Optimality of the weight function will be done later.

\bigskip 

\section{Density Theorem}
First we prove the density theorem as mentioned in the introduction section, as some part of the calculation with the sequence of test functions $u_\epsilon$ (see below) will be required for us in the proof of optimality part in Theorem \ref{mainresult}.

\begin{theorem}
    \label{dense} Let $sp=1, p=2$ then the following sequence of function converges to constant function $1$ in $W^{s,p}(0,1)$:
     \begin{equation}\label{uepsi}
u_\epsilon(x) = 
       \begin{cases}
      \frac{|\log \epsilon |}{|\log x|},  \hspace{8mm} x \in (0,\epsilon)\\
      1, \hspace{13mm}  x \in (\epsilon, 1-\epsilon),\\
       \frac{|\log (1-\epsilon)|}{|\log (1-x)|}, \hspace{3mm}x \in (1-\epsilon, 1).
       \end{cases}
\end{equation}
\end{theorem}
\begin{proof}
We want to show that $$[u_\epsilon - 1]_{s,p,(0,1)} + |u_\epsilon - 1|_{p,(0,1)} \rightarrow 0 \  \textrm{as} \ \epsilon \rightarrow 0.$$  
$|u_\epsilon - 1|_{p,(0,1)} \rightarrow 0$ follows from application Lebesgue dominated convergence theorem. Since $[u_\epsilon - 1]_{s,p,(0,1)}=[u_\epsilon]_{s,p,(0,1)}$, it is sufficient to show $[u_\epsilon ]_{s,p,(0,1)}  \rightarrow 0$. Using change of variable, definition and symmetry of the $u_\epsilon$, we can write 
$$[u_\epsilon ]_{s,p,(0,1)} = 2[u_\epsilon ]_{s,p,(0,\epsilon)} + 2 \int_0^\epsilon\int_{\epsilon}^{1-\epsilon}\frac{|u_\epsilon(x) -u_\epsilon(y)|^2}{|x-y|^2}dxdy := I_\epsilon + K_\epsilon. $$
Let us first estimate the term $I_\epsilon.$
Change the variable as $\log x =-s, \ \log y = -t$, to get
\begin{multline}
I_{\epsilon} = (\log\epsilon)^2 \int_{
|log \epsilon |}^\infty \int_{|\log\epsilon |}^\infty \frac{(s-t)^2 }{(st)^2 (e^{-s} -e^{-t})^2}e^{-s}e^{-t} dsdt \\ =  (\log\epsilon)^2 \int_{
|log \epsilon |}^\infty \int_{|\log\epsilon |}^\infty \frac{(s-t)^2 }{(st)^2 (e^s -e^t)^2}e^se^t dsdt 
= (\log\epsilon)^2 \int_{
|log \epsilon|}^\infty \frac{1}{t^2} \int_{
|log \epsilon|}^\infty\frac{1}{s^2}\frac{(s-t)^2}{ (e^{s-t} + e^{t-s} -2)}ds dt\\
= (\log\epsilon)^2 \int_{
|log \epsilon|}^\infty \frac{1}{t^2} \int_{
|log \epsilon|}^t\frac{1}{s^2}\frac{(s-t)^2}{ (e^{s-t} + e^{t-s} -2)}ds dt + (\log\epsilon)^2 \int_{
|log \epsilon|}^\infty \frac{1}{t^2} \int_t^{
|log \epsilon|}\frac{1}{s^2}\frac{(s-t)^2}{ (e^{s-t} + e^{t-s} -2)}dsdt \\
:= I_1^\epsilon+I_2^\epsilon.
\end{multline}
Now introducing the change of variable $s-t = x$ for a fixed $t$ in $I_2$ first to get 
\begin{equation*}\label{ffd}
I_2^\epsilon =  (\log\epsilon)^2 \int_{
|log \epsilon|}^\infty \frac{1}{t^2} \left(\int_0^{
\infty}\frac{g(x)}{(t+x)^2}dx\right)ds,
\end{equation*}
where $g(x) = \frac{x^2}{e^x+e^{-x}-2}.$ Note that  $g(x) \rightarrow 1$ as $x \rightarrow 0$ and $g(x) \leq Cx^2e^{-x}$. This implies that  $$\int_0^\infty g(x)dx <\infty.$$ Therefore 
$$I_2^\epsilon \leq C (\log\epsilon)^2 \int_{
|log \epsilon|}^\infty \frac{dt}{t^4} = \frac{C}{|\log\epsilon|} \rightarrow 0.$$
Using the same change of variable for $I_1^\epsilon$ and $t+x \geq |\log\epsilon|$ we obtain 
\begin{equation*}\label{i1e}
I_1^\epsilon = (\log\epsilon)^2 \int_{
|log \epsilon|}^\infty \frac{1}{t^2} \left(\int_{|\log\epsilon| -t}^{0}\frac{g(x)}{(t+x)^2}dx\right)dt \leq  \int_{
|log \epsilon|}^\infty \frac{1}{t^2} \left(\int_{|\log\epsilon| -t}^{0} g(x)dx\right)dt.
\end{equation*}
Since $g$ is an even function and we already know that $\int_0^\infty g(x)dx <\infty$, we have finally 
$$I_1^\epsilon \leq C \int_{
|log \epsilon|}^\infty \frac{dt}{t^2} = \frac{C}{|\log\epsilon|} \rightarrow 0 ,\  \textrm{as} \ \epsilon \rightarrow 0.$$
Now let us turn to the estimate of $K_\epsilon.$
\begin{multline}
K_\epsilon =  \int_0^\epsilon \int_\epsilon^{1-\epsilon} \frac{(\log \epsilon -\log x)^2}{\log^2x (x-y)^2 }dydx=  \int_0^\epsilon  \frac{(\log \epsilon -\log x)^2}{\log^2x}  \left[\frac{1}{x-y}\right]_\epsilon^{1-\epsilon} dx\\=\int_0^\epsilon  \frac{(\log \epsilon -\log x)^2}{\log^2x} \left\{\frac{1}{x-1+\epsilon} -\frac{1}{x-\epsilon}\right\} dx
=(1-2\epsilon) \int_0^\epsilon  \frac{(\log \epsilon -\log x)^2}{(\log^2x) (\epsilon -x)(1-x -\epsilon)} dx.
\end{multline}
Note that $1-x -\epsilon \geq 1-2\epsilon$,  this implies 
$$K_\epsilon \leq   \int_0^\epsilon  \frac{(\log \epsilon -\log x)^2}{(\log^2x) (\epsilon -x)}dx.$$
Change the variable as $\log x =-s$  to get
\begin{eqnarray*}
K_\epsilon \leq  \int_
{|\log\epsilon|}^\infty \frac{(\log\epsilon + s)^2}{s^2 (\epsilon - e^{-s})}e^{-s}ds = \int_
{|\log\epsilon|}^\infty \frac{(\log\epsilon + s)^2}{s^2 (\epsilon e^s- 1)}ds.
\end{eqnarray*}
Now again change  the variable as   $t = s +\log\epsilon$, to get 
$$K_\epsilon \leq \int_0^\infty \frac{t^2}{(t-\log \epsilon)^2(e^t-1)}dt= \int_0^\infty \frac{t^2}{(t+|\log \epsilon|)^2(e^t-1)}dt \leq \frac{1}{\log^2\epsilon} \int_0^\infty\frac{t^2}{(e^t-1)} dt \rightarrow 0. $$
The proof of the theorem is completed after noticing that the last integral is finite.
\end{proof}

\noindent\textbf{Optimality of the weight function:} \ We just treat the case $p=2$, the case of general $p$ can be treated with usual estimates and considering the same sequence of test function $u_\epsilon$.  Consider the sequence of function $u_\epsilon$ as in \eqref{uepsi}.  We have already  proved in Theorem \ref{dense} (after combining all the estimates therein), that

\begin{equation*}\label{ord}\int_0^1\int_0^1  \frac{(u_\epsilon(x) -u_\epsilon(y))^2}{|x-y|^{2}}dxdy \leq     \frac{C}{|\log\epsilon|} + o\left(\frac{1}{|\log\epsilon|}\right).\end{equation*}
For our notation $o(\delta)$ is a term such that $\frac{o(\delta)}{\delta}\rightarrow 0.$ 
Now we claim that the inequality in the above step is actually an equality. To prove this it is sufficient to prove that 
$$I_1^\epsilon, \  I_2^\epsilon \geq \frac{C}{|\log \epsilon|}.$$
But this is true because 
$$I_2^\epsilon \geq  (\log\epsilon)^2 \int_{
|log \epsilon|}^\infty \frac{1}{t^2} \left(\int_0^{
t}\frac{g(x)}{(t+x)^2}dx\right)dt \geq  (\log\epsilon)^2/4 \int_{
|log \epsilon|}^\infty \frac{1}{t^4}\left(\int_0^1 g(x) dx\right)dt. $$
In the last step we have used $t+x \leq 2t$ first and then $g(x) \geq 0$ and $t \geq |\log(\epsilon)| \geq 1$. Using similar estimates the other inequality also follows. Now let us calculate the  following expression: $$\int_0^1 \frac{(u_\epsilon- \bar{u_\epsilon})^2}{\delta_x (\log\delta_x)^2}dx = \int_0^1 \frac{(u_\epsilon- \bar{u_\epsilon})^2}{ \min\{x, 1-x\}\log^2(\min\{x,1-x \})}dx,$$ where  $$ \bar{u_\epsilon} = (u_\epsilon)_{(0,1)}.$$
First  let us calculate the term 
\begin{equation*}\label{bar}
\bar{u_\epsilon} =  2\int_0^\epsilon \frac{|\log\epsilon |}{| \log x |} dx+ \int_{\epsilon}^{1-\epsilon}dx = 2|\log\epsilon |\int_0^\epsilon \frac{dx}{|\log x |} + 1-2\epsilon  = 1 +C_\epsilon.
\end{equation*}
Also  note  that  the first term  $C_\epsilon
$ goes  faster to zero than $\epsilon$, after making an estimate $|\log x | \geq  |\log \epsilon |.$ Now  let us first estimate  the term 
\begin{multline}
\int_0^\epsilon \frac{(u_\epsilon - \bar{u_\epsilon})^2}{x \log^2x} dx=  \int_0^\epsilon \frac{u_\epsilon^2}{x \log^2x}dx+ \int_0^\epsilon \frac{\bar{u_\epsilon}^2}{x \log^2x}dx -2\int_0^\epsilon \frac{u_\epsilon  \bar{u_\epsilon}}{x \log^2x}dx \nonumber\\
 =\log^2\epsilon\int_0^\epsilon \frac{dx}{x \log^4x} + \bar{u_\epsilon}^2\int_0^\epsilon \frac{dx}{x \log^2x}  -2 \bar{u_\epsilon}\log\epsilon\int_0^\epsilon\frac{dx}{x\log^3x}.
\end{multline}
After making  change of the variable as $\log x = z$, we get finally after ``actual intergration",
$$\int_0^\epsilon \frac{(u_\epsilon - \bar{u_\epsilon})^2}{x \log^2x}dx = 
\frac{1}{3|\log\epsilon|} + \frac{\bar{u_\epsilon}^2}{|\log\epsilon|} - \frac{\bar{u_\epsilon}}{|\log\epsilon |} =  \frac{1}{3|\log\epsilon|} + o\left(\frac{1}{|\log \epsilon|}\right).$$
Now using \eqref{bar} we  will estimate   the following term (we estimate until $\frac{1}{2}$ as the other half can be estimated in similar way):
\begin{equation*}
\int_\epsilon^{.5} \frac{(u_\epsilon -\bar{u_\epsilon})^2}{x\log^2x}dx =   C_\epsilon^2 \int_\epsilon^{\frac{1}{2}} \frac{dx}{x\log^2x} = K C_\epsilon^2.
\end{equation*}
Note from $C_\epsilon^2 \rightarrow 0$ faster than $\epsilon^2$ (faster than $\frac{1}{|log\epsilon|}$), hence this term is of no importance. Therefore we have 
\begin{equation*}\label{oder}
\int_0^1 \frac{(u_\epsilon -\bar{u_\epsilon})^2}{\delta_x\log^2\delta_x}dx = \frac{1}{3|\log\epsilon|} + o\left(\frac{1}{|\log \epsilon|}\right).
\end{equation*}\smallskip 

\noindent Assume $\tau = p =2$. Now suppose that there exist a function, $f$ with the property that $f(x) \rightarrow \infty$ as $x \rightarrow 0+$ and there is a further improvement of the inequality in Theorem \ref{mainresult} in the sense that the following holds:
 \begin{align*}\label{neew}
            \left\| \frac{f(x)|\delta_x|^\gamma (u- (u)_{(0,1)})}{\log(2/\delta_x)}\right\|_{L^\tau((0,1))}\leq C[u]_{W^{s,p}((0,1))}, \ \ \forall\,u\in W^{s,p}((0,1)).
        \end{align*}
Consider the sequence of function $u_\epsilon$ as given in theorem \ref{dense}. Using similar type of computation as done above, it is easy to show that 
\begin{equation}\label{fast}
\int_0^1 \frac{(u_\epsilon -\bar{u_\epsilon})^2}{\delta_x\log^2\delta_x}dx \geq \frac{f(\epsilon)}{3|\log\epsilon|} + o\left(\frac{f(\epsilon)}{|\log \epsilon|}\right).\end{equation} 
This implies  as $\epsilon \rightarrow 0$,
$$ [u_\epsilon]_{1/2, 2, (0,1)}\left(\int_0^1 \frac{(u_\epsilon -\bar{u_\epsilon})^2}{\delta_x\log^2\delta_x}dx \right)^{-1} \leq \frac{C}{f(\epsilon)}\rightarrow 0,$$
and therefore clearly \eqref{fast} cannot be true. Similar argument will hold if $f \rightarrow \infty$ as $x \rightarrow 1-$.
\bigskip

\noindent\textbf{Acknowledgement:} \ Research of the third author is supported by Core Research Grant under project number CRG/2022/007867. Part of this work is carried out when the third author was visitng The Institute of Mathematical Sciences (IMSc). Purbita Jana and Prosenjit Roy would like to thank IMSc for all the facilities provided to them during their respective stay.  Prosenjit Roy would   also like to thank  Mr. Subhajit Roy for presenting him the proof of CKN inequality in \cite{Nguyen2018} and Mr. Vivek Sahu  for  pointing out some typos when the manuscript was in its earlier version.

\end{document}